\documentclass[reqno,11pt]{article} 
\usepackage[margin=1in]{geometry}  
\usepackage{amsmath, amssymb}
\usepackage{amsthm} 
\newcommand{\Mod}[1]{\ (\textup{mod}\ #1)}
\theoremstyle{plain} 
\newtheorem{theorem}{\indent\sc Theorem}[section]
\newtheorem{lemma}[theorem]{\indent\sc Lemma}

\newtheorem{proposition}[theorem]{\indent\sc Proposition}

\theoremstyle{definition} 
\newtheorem{definition}[theorem]{\indent\sc Definition}
\newtheorem{assumption}[theorem]{\indent\sc Assumption}
\newtheorem{conjecture}[theorem]{\indent\sc Conjecture}
\newtheorem{remark}[theorem]{\indent\sc Remark}
\newtheorem{example}[theorem]{\indent\sc Example}

\makeatletter
\def\address#1#2{\begingroup
\noindent\parbox[t]{7.8cm}{%
\small{\scshape\ignorespaces#1}\par\vskip1ex
\noindent\small{\itshape E-mail address}%
\/: #2\par\vskip4ex}\hfill%
\endgroup}%
\makeatother
\title{On Siegel invariants of certain CM-fields} 
\author{\textsc{Ja Kyung Koo, Gilles Robert, Dong Hwa Shin and Dong Sung Yoon$^*$} 
}
\date{} 
\begin{document}

\allowdisplaybreaks

\maketitle

\footnote{$^*$Corresponding author.}
\footnote{
\thanks{
The third author was supported by the National Research Foundation of Korea (NRF)
grant funded by the Korea government (MSIP) (2017R1A2B1006578), and by Hankuk
University of Foreign Studies Research Fund of 2018. 
The fourth (corresponding) author was
supported by Basic Science Research Program through the National Research Foundation of
Korea(NRF) funded by the Ministry of Education (2017R1D1A1B03030015).

} }

\begin{abstract}
We first construct Siegel invariants of some CM-fields
in terms of special values of theta constants,
which would be a generalization of Siegel-Ramachandra invariants of imaginary quadratic fields.
And, we further describe Galois actions on these invariants and provide some numerical examples to show that this invariant really generates the ray class field of a CM-field.
\end{abstract}

\section {Introduction}

Let $K$ be a number field and $\mathcal{O}_K$ be its ring of integers.
For a proper nontrivial ideal $\mathfrak{f}$ of $\mathcal{O}_K$
we denote by $\mathrm{Cl}(\mathfrak{f})$ and $K_\mathfrak{f}$ the ray class group of $K$ modulo $\mathfrak{f}$ and its corresponding ray class field, respectively (see \cite{Janusz}).
Suppose that there is a family
$\{\Psi_\mathfrak{f}(C)\}_{C\in\mathrm{Cl}(\mathfrak{f})}$
of algebraic numbers, which we shall call a Siegel family,
such that
\begin{itemize}
\item[(R1)] each $\Psi_\mathfrak{f}(C)$ belongs to $K_\mathfrak{f}$,
\item[(R2)] $\Psi_\mathfrak{f}(C)^{\sigma_\mathfrak{f}(D)}=\Psi_\mathfrak{f}(CD)$
for all $D\in\mathrm{Cl}(\mathfrak{f})$, where $\sigma_\mathfrak{f}:\mathrm{Cl}(\mathfrak{f})\rightarrow
\mathrm{Gal}(K_\mathfrak{f}/K)$ is the Artin reciprocity map for $\mathfrak{f}$.
\end{itemize}
Then, as is well known, every algebraic number $\Psi_\mathfrak{f}(C)$ becomes a primitive generator of $K_\mathfrak{f}$ over $K$ 
if
\begin{equation}\label{Kronecker limit}
\sum_{C\in\mathrm{Cl}(\mathfrak{f})}\chi(C)\ln
|\Psi_\mathfrak{f}(C)|\neq 0
\end{equation}
for any nontrivial character $\chi$ of $\mathrm{Cl}(\mathfrak{f})$ (\cite[Theorem 3 in Chapter 22]{Lang87}), which motivates this paper.
In particular, if $K$ is an imaginary quadratic field, then
the Siegel-Ramachandra invariants form such a Siegel family having the properties (R1) and (R2) (see $\S$\ref{introduce}).
Furthermore, they also satisfy (\ref{Kronecker limit}) in most cases by the second Kronecker limit formula.
These invariants are defined by the special values of certain modular units, Siegel functions,  which can be expressed in terms of theta constants (Remark \ref{generalization} (ii)).
\par
Let $K$ be a CM-field and $\mathfrak{f}$ be a proper nontrivial ideal of $\mathcal{O}_K$
satisfying certain conditions (Assumption \ref{pp}).
In this paper, we shall first construct a meromorphic Siegel modular function of level $N$ ($\geq2$), which would be a multi-variable generalization of the Siegel function,
by making use of theta constants (Definition \ref{bigdef} and Proposition \ref{bigtheta}).
Furthermore, we shall assign
the special value $\Theta_\mathfrak{f}(C)$ of this function to each ray class $C$ in $\mathrm{Cl}(\mathfrak{f})$,
and call it the Siegel invariant modulo $\mathfrak{f}$ at $C$ (Definition \ref{defSiegel}).
This value depends only on $\mathfrak{f}$ and the class $C$ (Propositions \ref{indep2} and \ref{indep1}),
essentially by the fact that the Siegel modular variety is a moduli space for principally polarized abelian varieties. 
Finally, we are able to show by applying Shimura's reciprocity law  that the Siegel invariant
$\Theta_\mathfrak{f}(C)$, as a possible ray class invariant (Conjecture \ref{conj}), satisfies the transformation formula (R2), that is,
\begin{equation*}
\Theta_\mathfrak{f}(C)^{\sigma_\mathfrak{f}(D)}
=\Theta_\mathfrak{f}(CD)\quad\textrm{for all}~D\in\mathrm{Cl}(\mathfrak{f})
\end{equation*}
(Theorem \ref{main}).
By making use of Maple software we also present a numerical example with a non-imaginary quadratic CM-field $K$ for which $K_\mathfrak{f}$ is generated by $\Theta_\mathfrak{f}(C)$ over $K$ (Example \ref{explicit example}).

\section {Siegel-Ramachandra invariants}\label{introduce}

Let $\mathbf{v}=\left[\begin{matrix}r\\s\end{matrix}\right]\in\mathbb{Q}^2\setminus\mathbb{Z}^2$.
The Siegel function
$g_\mathbf{v}(\tau)$
on the complex upper half-plane $\mathbb{H}=\{\tau\in\mathbb{C}~|~\mathrm{Im}(\tau)>0\}$ is given by the infinite product
\begin{equation}\label{Siegel}
g_\mathbf{v}(\tau)=-q^{\mathbf{B}_2(r)/2}
e^{\pi\mathrm{i}s(r-1)}(1-q^{r}e^{2\pi\mathrm{i}s})\prod_{n=1}^\infty
(1-q^{n+r}e^{2\pi\mathrm{i}s})(1-q^{n-r}e^{-2\pi\mathrm{i}s}),
\end{equation}
where $\mathbf{B}_2(x)=x^2-x+1/6$ is the second Bernoulli polynomial and
$q=e^{2\pi\mathrm{i}\tau}$.
If $N$ ($\geq2$) is a positive integer so that $N\mathbf{v}\in\mathbb{Z}^2$, then $g_\mathbf{v}(\tau)^{12N}$
is a meromorphic modular function of level $N$
which has neither a zero nor a pole on $\mathbb{H}$ (\cite[Theorem 1.2 in Chapter 2]{K-L}).
\par
Let $K$ be an imaginary quadratic field. Let $\mathfrak{f}$ be a proper nontrivial ideal of $\mathcal{O}_K$ in which $N$ is the smallest positive integer, and let $C\in\mathrm{Cl}(\mathfrak{f})$. Take any integral ideal $\mathfrak{c}$ in $C$, and
let $\omega_1,\omega_2\in\mathbb{C}$ and $a,b\in\mathbb{Z}$ such that
\begin{eqnarray*}
\mathfrak{f}\mathfrak{c}^{-1}&=&\mathbb{Z}\omega_1+
\mathbb{Z}\omega_2\quad\textrm{with}~\omega=\omega_1/\omega_2\in\mathbb{H},\\
N&=&a\omega_1+b\omega_2.
\end{eqnarray*}
The Siegel-Ramachandra invariant $g_\mathfrak{f}(C)$ modulo $\mathfrak{f}$ at $C$ is defined as
\begin{equation}\label{S-R}
g_\mathfrak{f}(C)=g_{\left[\begin{smallmatrix}a/N\\b/N\end{smallmatrix}\right]}
(\omega)^{12N}.
\end{equation}
This value depends only on $\mathfrak{f}$ and the class $C$, not on the choices of $\mathfrak{c}$ and $\omega_1,\omega_2$(\cite[Chapter 11, $\S$1]{K-L}). Furthermore, it lies in $K_\mathfrak{f}$ and satisfies
\begin{equation*}
g_\mathfrak{f}(C)^{\sigma_\mathfrak{f}(D)}=g_\mathfrak{f}(CD)\quad(D\in\mathrm{Cl}(\mathfrak{f})) \end{equation*}
(\cite[Theorem 1.1 in Chapter 11]{K-L}).

\begin{proposition}
Let $\mathfrak{f}=\prod_{\mathfrak{p}\,|\,\mathfrak{f}}\mathfrak{p}^{e_\mathfrak{p}}$
be the prime ideal factorization of $\mathfrak{f}$, and let
\begin{equation*}
G_\mathfrak{p}=(\mathcal{O}_K/\mathfrak{p}^{e_\mathfrak{p}})^\times/
\{\mu+\mathfrak{p}^{e_\mathfrak{p}}~|~\mu\in\mathcal{O}_K^\times\}.
\end{equation*}
If $|G_\mathfrak{p}|>2$ for every $\mathfrak{p}\,|\,\mathfrak{f}$, then
any nonzero power of $g_\mathfrak{f}(C)$ generates $K_\mathfrak{f}$ over $K$.
\end{proposition}
\begin{proof}
See \cite[Theorem 4.6]{K-Y2}.
\end{proof}

\begin{remark}
\begin{itemize}
\item[(i)] This result is obtained by utilizing the second Kronecker limit formula (\cite[Theorem 9 in Chapter II]{Siegel} or \cite[Theorem 2 in Chapter 22]{Lang87}).
\item[(ii)] Without any condition, Ramachandra (\cite{Ramachandra}) constructed
a primitive generator of $K_\mathfrak{f}$ over $K$ as
a high power product of Siegel-Ramachandra invariants and
singular values of the modular discriminant $\Delta$-function.
\end{itemize}

\end{remark}

\section {Actions on Siegel modular functions}

In this section we shall briefly recall the action of an idele group on the field of meromorphic Siegel modular functions
due to Shimura.
\par
For a positive integer $g$ and a commutative ring $R$ with unity, we let
\begin{eqnarray*}
\mathrm{GSp}_{2g}(R)&=&\left\{\alpha\in\mathrm{GL}_{2g}(R)~|~\alpha^TJ\alpha=\nu J~
\textrm{for some}~\nu\in R^\times\right\}~\textrm{where}~J=\left[\begin{matrix}O&-I_g\\
I_g&O\end{matrix}\right],\\
\mathrm{Sp}_{2g}(R)&=&\left\{\alpha\in\mathrm{GL}_{2g}(R)~|~\alpha^TJ\alpha=J\right\}.
\end{eqnarray*}
Here, $\alpha^T$ stands for the transpose of the matrix $\alpha$.
Observe that the relation $\alpha^T J\alpha=\nu J$ implies $\det(\alpha)=\nu^g$
(\cite[(1.11)]{Shimura00}). 
If $\alpha$ belongs to either $\mathrm{GSp}_{2g}(R)$ or $\mathrm{Sp}_{2g}(R)$, then
$\alpha^T$ also belongs to the same group (\cite[p. 17]{Shimura00}).
\par
The symplectic group $\mathrm{Sp}_{2g}(\mathbb{Z})$ acts on the Siegel upper half-space
\begin{equation*}
\mathbb{H}_g=\left\{Z\in M_g (\mathbb{C})~|~Z^T=Z,~\mathrm{Im}(Z)~
\textrm{is positive definite}\right\}
\end{equation*}
by
\begin{equation}\label{fractional}
\gamma(Z)=
(AZ+B)(CZ+D)^{-1}\quad(\gamma\in\mathrm{Sp}_{2g}(\mathbb{Z}),~Z\in\mathbb{H}_g),
\end{equation}
where $A,B,C,D$ are $g\times g$ block matrices of $\gamma$
(\cite[Proposition 1 in $\S$1]{Klingen}).
For a positive integer $N$ let
\begin{equation*}
\Gamma(N)=\left\{\gamma\in\mathrm{Sp}_{2g}(\mathbb{Z})~|~\gamma\equiv I_{2g}
\Mod{N\cdot M_{2g}(\mathbb{Z})}\right\}.
\end{equation*}
We call a holomorphic function $f:\mathbb{H}_g\rightarrow\mathbb{C}$ a Siegel modular form
of weight $k$ and level $N$ if
\begin{itemize}
\item[\textrm{(M1)}] $f(\gamma(Z))=\det(CZ+D)^kf(Z)$ for every $\gamma=
\left[\begin{matrix}A&B\\C&D\end{matrix}\right]\in\Gamma(N)$,
\item[\textrm{(M2)}] $f$ is holomorphic at every cusp when $g=1$.
\end{itemize}
Every Siegel modular form $f$ can be expressed as
\begin{equation*}
f(Z)=\sum_\beta c(\beta)e(\mathrm{tr}(\beta Z)/N)\quad(c(\beta)\in\mathbb{C}),
\end{equation*}
where $\beta$ runs over all $g\times g$ positive semi-definite symmetric matrices over half integers with integer diagonal entries and $e(x)=e^{2\pi\mathrm{i}x}$ ($x\in\mathbb{R}$) (\cite[Theorem 1 in $\S$4]{Klingen}).
Here, we call $c(\beta)$ the Fourier coefficients of $f$.
For a subfield $F$ of $\mathbb{C}$ we set
\begin{eqnarray*}
\mathcal{M}_k(\Gamma(N),F)&=&
\textrm{the $F$-vector space of all Siegel modular forms of weight $k$ and level $N$}\\
&&\textrm{with Fourier coefficients in $F$},\\
\mathcal{M}_k(F)&=&\bigcup_{N=1}^\infty\mathcal{M}_k(\Gamma(N),F),\\
\mathcal{A}_0(\Gamma(N),F)&=&\textrm{the field of all meromorphic Siegel modular functions of the form~$g/h$,}\\
&&\textrm{with $g\in\mathcal{M}_k(F)$ and $h\in\mathcal{M}_k(F)\setminus\{0\}$ for some $k$,}\\
&&\textrm{which is invariant under the group $\Gamma(N)$},\\
\mathcal{A}_0(F)&=&\bigcup_{N=1}^\infty\mathcal{A}_0(\Gamma(N),F).
\end{eqnarray*}
In particular, let
\begin{eqnarray*}
\mathcal{F}_N&=&\mathcal{A}_0(\Gamma(N),\mathbb{Q}(\zeta_N)),\\
\mathcal{F}&=&A_0(\mathbb{Q}_\textrm{ab}),
\end{eqnarray*}
where $\zeta_N=e(1/N)$ and $K_\textrm{ab}$ denotes the maximal abelian extension of a number field $K$.
\par
On the other hand, we let
\begin{eqnarray*}
G&=&\mathrm{GSp}_{2g}(\mathbb{Q}),\\
G_+&=&\{\alpha\in G~|~\alpha^TJ\alpha=\nu J~\textrm{for some}~\nu>0\},\\
G_\mathbb{A}&=&\textrm{the adelization of $G$ with $G_0$ and $G_\infty$}\\
&&\textrm{the non-archimedean part and the archimedean part, respectively},\\
G_{\mathbb{A}+}&=&G_0G_{\infty+},~\textrm{where $G_{\infty+}$ is the identity component
of $G_\infty$}.
\end{eqnarray*}
Shimura presented in \cite[Theorem 8.10]{Shimura00} a group homomorphism
\begin{equation*}
\tau:G_\mathbb{A+}
\rightarrow\mathrm{Aut}(\mathcal{F})
\end{equation*}
satisfying the following properties: Let $f\in\mathcal{F}$.
\begin{itemize}
\item[\textup{(A1)}] $f^{\tau(\alpha)}=f\circ\alpha$ for all $\alpha\in G_+$, where $\alpha$ acts on $\mathbb{H}_g$ by the same way as in (\ref{fractional}).
\item[\textup{(A2)}] $f^{\tau(\iota(s))}=f^{[s,\mathbb{Q}]}$
for all $s\in\prod_p\mathbb{Z}_p^\times$,
where $\iota(s)=\left[\begin{matrix}I_g&O\\
O&s^{-1}I_g\end{matrix}\right]$.
Here, the action of $[s,\mathbb{Q}]$ on $f$ is understood as the
action of it on the Fourier coefficients of $f$ (see also \cite[Theorem 5]{Shimura75}).
\end{itemize}
Note that the mapping $s\mapsto[s,\mathbb{Q}]$
yields an isomorphism of $\prod_p\mathbb{Z}_p^\times$ onto
$\mathrm{Gal}(\mathbb{Q}_\textrm{ab}/\mathbb{Q})$
(\cite[$\S$8.1]{Shimura00}).
Then, $\mathcal{F}_N$ coincides with the fixed field of $\mathcal{F}$ by the subgroup
\begin{equation*}
\mathbb{Q}^\times\cdot\{\alpha\in G_{\mathbb{A}+}~|~\alpha_p\in\mathrm{GL}_{2g}(\mathbb{Z}_p)
~\textrm{and}~\alpha_p\equiv I_{2g}\Mod{N\cdot M_{2g}(\mathbb{Z}_p)}~
\textrm{for all rational primes $p$}\}
\end{equation*}
of $G_{\mathbb{A}+}$ (\cite[Theorem 3]{Shimura75} or \cite[Theorem 8.10 (6)]{Shimura00}).

\section {Siegel modular functions in terms of theta constants}

Let $g$ and $N$ be positive integers, and let $\mathbf{r},\mathbf{s}\in(1/N)\mathbb{Z}^{g\phantom{\big|}}$.
The (classical) theta constant $\theta(\left[\begin{smallmatrix}\mathbf{r}\\\mathbf{s}\end{smallmatrix}\right],Z)$
 is defined by
\begin{equation}\label{thetaconstant}
\theta(\left[\begin{smallmatrix}\mathbf{r}\\\mathbf{s}\end{smallmatrix}\right],Z)=
\sum_{\displaystyle\mathbf{n}\in\mathbb{Z}^{g}}e\left(
\frac{1}{2}(\mathbf{n}+\mathbf{r})^TZ(\mathbf{n}+\mathbf{r})+(\mathbf{n}+\mathbf{r})^T\mathbf{s}\right)
\quad(Z\in\mathbb{H}_g).
\end{equation}
\par
For a matrix $E\in M_g (\mathbb{Z})$ we mean by $\{E\}$
the $g$-vector whose components are the diagonal entries of $E$.

\begin{lemma}\label{Igusa}
We have the following properties of theta constants.
\begin{itemize}
\item[\textup{(i)}] $\theta(\left[\begin{smallmatrix}\mathbf{r}\\\mathbf{s}\end{smallmatrix}\right],Z)$ is
identically zero if and only if
$\mathbf{r},\mathbf{s}\in(1/2)\mathbb{Z}^{g\phantom{\big|}\hspace{-0.1cm}}$ and
$e(2\mathbf{r}^T\mathbf{s})=-1$.
\item[\textup{(ii)}]
$\theta(-\left[\begin{smallmatrix}\mathbf{r}\\\mathbf{s}\end{smallmatrix}\right],Z)=
\theta(\left[\begin{smallmatrix}\mathbf{r}\\\mathbf{s}\end{smallmatrix}\right],Z)$.
\item[\textup{(iii)}] If $\mathbf{a},\mathbf{b}\in\mathbb{Z}^{g\phantom{\big|}\hspace{-0.1cm}}$, then
$\theta(\left[\begin{smallmatrix}\mathbf{r}\\\mathbf{s}\end{smallmatrix}\right]
    +\left[\begin{smallmatrix}\mathbf{a}\\\mathbf{b}\end{smallmatrix}\right],Z)^N
    =\theta(\left[\begin{smallmatrix}\mathbf{r}\\\mathbf{s}\end{smallmatrix}\right],Z)^N$.
\item[\textup{(iv)}]
If $\gamma=\left[\begin{matrix}A&B\\C&D\end{matrix}\right]\in\mathrm{Sp}_{2g}(\mathbb{Z})$, then
\begin{equation*}
\theta(\left[\begin{smallmatrix}\mathbf{r}\\\mathbf{s}\end{smallmatrix}\right],\gamma(Z))^{4N}
=\kappa(\gamma)^{2N}\det(CZ+D)^{2N}
e\left(2N(\mathbf{r}^T\mathbf{s}-(\mathbf{r}')^T\mathbf{s}')\right)
\theta(\left[\begin{smallmatrix}\mathbf{r}'\\\mathbf{s}'\end{smallmatrix}\right]
+\tfrac{1}{2}\left[\begin{smallmatrix}\{A^TC\}\\\{B^TD\}\end{smallmatrix}\right],Z)^{4N},
\end{equation*}
where $\left[\begin{matrix}\mathbf{r}'\\\mathbf{s}'\end{matrix}\right]
=\gamma^T\left[\begin{matrix}\mathbf{r}\\\mathbf{s}\end{matrix}\right]$ and $\kappa(\gamma)$ is a $4$-th root of unity which depends only on $\gamma$.
\item[\textup{(v)}] The function $\theta(\left[\begin{smallmatrix}\mathbf{r}\\\mathbf{s}\end{smallmatrix}\right],Z)/
    \theta(\left[\begin{smallmatrix}\mathbf{0}\\\mathbf{0}\end{smallmatrix}\right],Z)$ belongs to $\mathcal{F}_{2N^2}$.
Furthermore, if $\alpha\in G_{\mathbb{A}+}\cap\prod_p\mathrm{GL}_{2g}(\mathbb{Z}_p)$ such that
$\alpha_p\equiv\left[\begin{matrix}I_g & O\\O&tI_g\end{matrix}\right]\Mod{2N^2\cdot
M_{2g}(\mathbb{Z}_p)}$ for all rational primes $p$ with a positive integer $t$, then
\begin{equation*}
\left(\frac{\theta(\left[\begin{smallmatrix}\mathbf{r}\\\mathbf{s}\end{smallmatrix}\right],Z)}
{\theta(\left[\begin{smallmatrix}\mathbf{0}\\\mathbf{0}\end{smallmatrix}\right],Z)}\right)^{\tau(\alpha)}=
\frac{\theta(\left[\begin{smallmatrix}\mathbf{r}\\t\mathbf{s}\end{smallmatrix}\right],Z)}
{\theta(\left[\begin{smallmatrix}\mathbf{0}\\\mathbf{0}\end{smallmatrix}\right],Z)}.
\end{equation*}
\end{itemize}
\end{lemma}
\begin{proof}
(i) See \cite[Theorem 2]{Igusa}.\\
(ii) This is immediate from the definition (\ref{thetaconstant}).\\
(iii) See \cite[p. 676 (13)]{Shimura76}.\\
(iv) See \cite[Proposition 1.3]{Shimura76} or \cite[Theta Transformation Formula 8.6.1]{B-L}.\\
(v) See \cite[Proposition 1.7]{Shimura76}.
\end{proof}

Let
\begin{eqnarray*}
S_-&=&\{\left[\begin{smallmatrix}\mathbf{a}\\\mathbf{b}\end{smallmatrix}\right]~|~
\mathbf{a},\mathbf{b}\in\{0,1/2\}^{g\phantom{\big|}\hspace{-0.1cm}}~\textrm{such that}~e(2\mathbf{a}^T\mathbf{b})=-1\},\\
S_+&=&\{\left[\begin{smallmatrix}\mathbf{c}\\\mathbf{d}\end{smallmatrix}\right]~|~
\mathbf{c},\mathbf{d}\in\{0,1/2\}^{g\phantom{\big|}\hspace{-0.1cm}}~\textrm{such that}~e(2\mathbf{c}^T\mathbf{d})=1)\}.
\end{eqnarray*}
By Lemma \ref{Igusa} (i) and (iv) one can regard each element $\gamma=\left[\begin{matrix}A&B\\C&D\end{matrix}\right]\in\mathrm{Sp}_{2g}(\mathbb{Z})$ as a permutation of the set $S_-$ (and $S_+$)
so that
\begin{equation*}
\left[\begin{matrix}
\mathbf{a}\\\mathbf{b}
\end{matrix}\right]\mapsto\gamma^T\left[\begin{matrix}
\mathbf{a}\\\mathbf{b}
\end{matrix}\right]+\frac{1}{2}\left[\begin{matrix}\{A^TC\}\\\{B^TD\}\end{matrix}\right]\Mod{\mathbb{Z}^{2g\phantom{\big|}\hspace{-0.1cm}}}.
\end{equation*}

\begin{definition}\label{bigdef}
We define a function
\begin{equation*}
\Theta(\left[\begin{smallmatrix}\mathbf{r}\\\mathbf{s}\end{smallmatrix}\right],Z)=
2^{4N}e\left(-2^gN(2^g-1)(2^g+1)\mathbf{r}^T\mathbf{s}\right)
\frac{\prod_{\left[\begin{smallmatrix}
 \mathbf{a}\\\mathbf{b}
\end{smallmatrix}\right]\in S_-}
\theta(\left[\begin{smallmatrix}\mathbf{a}\\\mathbf{b}\end{smallmatrix}\right]-
\left[\begin{smallmatrix}\mathbf{r}\\\mathbf{s}\end{smallmatrix}\right],Z)^{4N(2^g+1)}}
{\prod_{\left[\begin{smallmatrix}
 \mathbf{c}\\\mathbf{d}
\end{smallmatrix}\right]\in S_+}
\theta(\left[\begin{smallmatrix}\mathbf{c}\\\mathbf{d}\end{smallmatrix}\right],Z)^{4N(2^g-1)}}
\quad(Z\in\mathbb{H}_g).
\end{equation*}
\end{definition}

\begin{remark}\label{generalization}
\begin{itemize}
\item[(i)] One can easily check that 
\begin{equation*}
|S_-|=2^{g-1}(2^g-1)\quad\textrm{and}\quad
|S_+|=2^{g-1}(2^g+1).
\end{equation*}
Hence we have
\begin{equation*}
\mathrm{lcm}(|S_-|,|S_+|)=2^{g-1}(2^g-1)(2^g+1)=|S_-|(2^g+1)=|S_+|(2^g-1).
\end{equation*}
\item[(ii)] When $g=1$, let $N\geq2$ and $\left[\begin{matrix}r\\s\end{matrix}\right]
\in(1/N)\mathbb{Z}^2\setminus\mathbb{Z}^2$.
By using Jacobi's triple product identity (\cite[(17.3)]{Fine}) which reads
\begin{equation*}
\sum_{n\in\mathbb{Z}}a^n q^{n^2/2}=
\prod_{n=1}^\infty(1-q^{n})(1+aq^{n-1/2})(1+a^{-1}q^{n-1/2})\quad(a\in\mathbb{C}^\times),
\end{equation*}
one can justify
\begin{equation*}
\Theta(\left[\begin{smallmatrix}r\\s\end{smallmatrix}\right],\tau)
=g_{\left[\begin{smallmatrix}r\\s\end{smallmatrix}\right]}(\tau)^{12N}\quad(\tau\in\mathbb{H}).
\end{equation*}
This shows that the function in Definition \ref{bigdef} would be a multi-variable generalization of the Siegel function described in (\ref{Siegel}).
\end{itemize}
\end{remark}

\begin{lemma}\label{bigtransf}
We obtain the following transformation formulas for $\Theta(\left[\begin{smallmatrix}
\mathbf{r}\\\mathbf{s}\end{smallmatrix}\right],Z)$.
\begin{itemize}
\item[\textup{(i)}] $\Theta(\left[\begin{smallmatrix}\mathbf{a}\\\mathbf{b}\\\end{smallmatrix}\right]+
    \left[\begin{smallmatrix}\mathbf{r}\\\mathbf{s}\end{smallmatrix}\right],Z)=
    \Theta(\left[\begin{smallmatrix}\mathbf{r}\\\mathbf{s}\end{smallmatrix}\right],Z)$
    for all $\mathbf{a},\mathbf{b}\in\mathbb{Z}^{g\phantom{\big|}\hspace{-0.1cm}}$.
\item[\textup{(ii)}] $\Theta(\left[\begin{smallmatrix}\mathbf{r}\\\mathbf{s}\end{smallmatrix}\right],\gamma(Z))=
    \Theta(\gamma^T\left[\begin{smallmatrix}\mathbf{r}\\\mathbf{s}\end{smallmatrix}\right],Z)$
    for all $\gamma\in\mathrm{Sp}_{2g}(\mathbb{Z})$.
\item[\textup{(iii)}] If $\alpha\in G_{\mathbb{A}+}\cap\prod_p\mathrm{GL}_{2g}(\mathbb{Z}_p)$ for which
$\alpha_p\equiv\left[\begin{matrix}I_g & O\\O&tI_g\end{matrix}\right]\Mod{2N^2\cdot
M_{2g}(\mathbb{Z}_p)}$ for all rational primes $p$ with a positive integer $t$, then
$\Theta(\left[\begin{smallmatrix}
\mathbf{r}\\\mathbf{s}\end{smallmatrix}\right],Z)^{\tau(\alpha)}=
\Theta(\left[\begin{smallmatrix}
\mathbf{r}\\t\mathbf{s}\end{smallmatrix}\right],Z)$.
\end{itemize}
\end{lemma}
\begin{proof}
(i) This is immediate from Lemma \ref{Igusa} (iii) and the Definition \ref{bigdef}.\\
(ii) Let $\gamma=\left[\begin{matrix}A&B\\C&D\end{matrix}\right]\in\mathrm{Sp}_{2g}(\mathbb{Z})$.
For $\mathbf{x},\mathbf{y}\in\mathbb{Q}^g$, we set $\left[\begin{matrix}\mathbf{x}'\\\mathbf{y}'\end{matrix}\right]
=\gamma^T\left[\begin{matrix}\mathbf{x}\\\mathbf{y}\end{matrix}\right]=\left[\begin{matrix}A^T\mathbf{x}+C^T\mathbf{y}\\ B^T\mathbf{x}+D^T\mathbf{y}\end{matrix}\right]$.
Here we observe that
\begin{eqnarray*}
&&\frac{\prod_{\left[\begin{smallmatrix}
 \mathbf{a}\\\mathbf{b}
\end{smallmatrix}\right]\in S_-}
e\left(2N(2^g+1)((\mathbf{a}-\mathbf{r})^T(\mathbf{b}-\mathbf{s})-(\mathbf{a}'-\mathbf{r}')^T(\mathbf{b}'-\mathbf{s}'))\right)
}
{\prod_{\left[\begin{smallmatrix}
 \mathbf{c}\\\mathbf{d}
\end{smallmatrix}\right]\in S_+}
e\left(2N(2^g-1)(\mathbf{c}^T\mathbf{d}-(\mathbf{c}')^T\mathbf{d}')\right)
}\\
&=&e\left(2^gN(2^g-1)(2^g+1)(\mathbf{r}^T\mathbf{s}-(\mathbf{r}')^T\mathbf{s}')\right)
\frac{\prod_{\left[\begin{smallmatrix}
 \mathbf{a}\\\mathbf{b}
\end{smallmatrix}\right]\in S_-}
e\left(2N(\mathbf{a}^T\mathbf{b}-(\mathbf{a}')^T\mathbf{b}')\right)
}
{\prod_{\left[\begin{smallmatrix}
 \mathbf{c}\\\mathbf{d}
\end{smallmatrix}\right]\in S_+}
e\left(-2N(\mathbf{c}^T\mathbf{d}-(\mathbf{c}')^T\mathbf{d}')\right)}
\\
&=&e\left(2^gN(2^g-1)(2^g+1)(\mathbf{r}^T\mathbf{s}-(\mathbf{r}')^T\mathbf{s}')\right)
\frac{\prod_{\left[\begin{smallmatrix}
 \mathbf{a}\\ \mathbf{b}
\end{smallmatrix}\right]\in S_-}
e\left(-2N(\mathbf{a}^TAB^T\mathbf{a}+\mathbf{b}^TCD^T\mathbf{b})\right)
}{\prod_{\left[\begin{smallmatrix}
 \mathbf{c}\\ \mathbf{d}
\end{smallmatrix}\right]\in S_+}
e\left(2N(\mathbf{c}^TAB^T\mathbf{c}+\mathbf{d}^TCD^T\mathbf{d})\right)}\\
&&\textrm{because $AD^T+BC^T=I_g+2BC^T$}\\
&=&e\left(2^gN(2^g-1)(2^g+1)(\mathbf{r}^T\mathbf{s}-(\mathbf{r}')^T\mathbf{s}')\right)
\prod_{\left[\begin{smallmatrix}
 \mathbf{a}\\ \mathbf{b}
\end{smallmatrix}\right]\in \{0,1/2\}^{2g}}
e\left(-2N(\mathbf{a}^TAB^T\mathbf{a}+\mathbf{b}^TCD^T\mathbf{b})\right)\\
&&\textrm{because $S_+=\{0,1/2\}^{2g}\setminus S_-$}\\
&=&e\left(2^gN(2^g-1)(2^g+1)(\mathbf{r}^T\mathbf{s}-(\mathbf{r}')^T\mathbf{s}')\right)
\prod_{ \mathbf{a}\in \{0,1/2\}^{g}}
e\left(-2^{g+1}N(\mathbf{a}^T(AB^T+CD^T)\mathbf{a})\right)\\
&=&e\left(2^gN(2^g-1)(2^g+1)(\mathbf{r}^T\mathbf{s}-(\mathbf{r}')^T\mathbf{s}')\right).
\end{eqnarray*}
Hence we derive
\begin{eqnarray*}
\Theta(\left[\begin{smallmatrix}\mathbf{r}\\\mathbf{s}\end{smallmatrix}\right],\gamma(Z))&=&
2^{4N}e\left(-2^gN(2^g-1)(2^g+1)\mathbf{r}^T\mathbf{s}\right)
\frac{\prod_{\left[\begin{smallmatrix}
 \mathbf{a}\\\mathbf{b}
\end{smallmatrix}\right]\in S_-}
\theta(\left[\begin{smallmatrix}\mathbf{a}\\\mathbf{b}\end{smallmatrix}\right]-
\left[\begin{smallmatrix}\mathbf{r}\\\mathbf{s}\end{smallmatrix}\right],\gamma(Z))^{4N(2^g+1)}}
{\prod_{\left[\begin{smallmatrix}
 \mathbf{c}\\\mathbf{d}
\end{smallmatrix}\right]\in S_+}
\theta(\left[\begin{smallmatrix}\mathbf{c}\\\mathbf{d}\end{smallmatrix}\right],\gamma(Z))^{4N(2^g-1)}}\\
&=&
2^{4N}e\left(-2^gN(2^g-1)(2^g+1)((\mathbf{r}')^T\mathbf{s}')\right)\\
&&\times
\frac{\prod_{\left[\begin{smallmatrix}
 \mathbf{a}\\\mathbf{b}
\end{smallmatrix}\right]\in S_-}
\theta(\gamma^T\left[\begin{smallmatrix}\mathbf{a}\\\mathbf{b}\end{smallmatrix}\right]
+\tfrac{1}{2}\left[\begin{smallmatrix}\{A^TC\}\\\{B^TD\}\end{smallmatrix}\right]
-\left[\begin{smallmatrix}
\mathbf{r}'\\\mathbf{s}'
\end{smallmatrix}\right],Z
)^{4N(2^g+1)}}
{\prod_{\left[\begin{smallmatrix}
 \mathbf{c}\\\mathbf{d}
\end{smallmatrix}\right]\in S_+}
\theta(\gamma^T\left[\begin{smallmatrix}\mathbf{c}\\\mathbf{d}\end{smallmatrix}\right]
+\tfrac{1}{2}\left[\begin{smallmatrix}\{A^TC\}\\\{B^TD\}\end{smallmatrix}\right],Z
)^{4N(2^g-1)}}\quad\textrm{by Lemma \ref{Igusa} (iv)}\\
&=&
2^{4N}e\left(-2^gN(2^g-1)(2^g+1)((\mathbf{r}')^T\mathbf{s}')\right)
\frac{ \prod_{\left[\begin{smallmatrix}
 \mathbf{a}\\\mathbf{b}
\end{smallmatrix}\right]\in S_-}
\theta(\left[\begin{smallmatrix}\mathbf{a}\\\mathbf{b}\end{smallmatrix}\right]-
\left[\begin{smallmatrix}\mathbf{r}'\\\mathbf{s}'\end{smallmatrix}\right],Z)^{4N(2^g+1)}}
{\prod_{\left[\begin{smallmatrix}
 \mathbf{c}\\\mathbf{d}
\end{smallmatrix}\right]\in S_+}
\theta(\left[\begin{smallmatrix}\mathbf{c}\\\mathbf{d}\end{smallmatrix}\right],Z)^{4N(2^g-1)}}\\
&&\textrm{by Lemma \ref{Igusa} (iii) and the fact that $\gamma$ is a permutation of $S_-$ (and $S_+$)}\\
&=&\Theta(\left[\begin{smallmatrix}\mathbf{r}'\\\mathbf{s}'\end{smallmatrix}\right],Z).
\end{eqnarray*}
(iii) Since $t$ is odd, $\left[\begin{matrix}\mathbf{a}\\
\mathbf{b}\end{matrix}\right]\mapsto
\left[\begin{matrix}\mathbf{a}\\
t\mathbf{b}\end{matrix}\right]\Mod{\mathbb{Z}^{2g\phantom{\big|}\hspace{-0.1cm}}}$ gives rise to a permutation of $S_-$ (and $S_+$). 
Furthermore, it follows from \cite[\S 8.1]{Shimura00} that
\begin{equation*}
e(1/N)^{\tau(\alpha)}=e(t/N).
\end{equation*}
Hence we see by Lemma \ref{Igusa} (v) that
 \begin{equation*}
 \begin{array}{rcl}
\Theta(\left[\begin{smallmatrix}\mathbf{r}\\\mathbf{s}\end{smallmatrix}\right],Z)^{\tau(\alpha)}&=&
2^{4N}e\left(-2^gtN(2^g-1)(2^g+1)\mathbf{r}^T\mathbf{s}\right)
\displaystyle\frac{ \prod_{\left[\begin{smallmatrix}
 \mathbf{a}\\\mathbf{b}
\end{smallmatrix}\right]\in S_-}
\theta(\left[\begin{smallmatrix}\mathbf{a}\\t\mathbf{b}\end{smallmatrix}\right]-
\left[\begin{smallmatrix}\mathbf{r}\\t\mathbf{s}\end{smallmatrix}\right],Z)^{4N(2^g+1)}}
{\prod_{\left[\begin{smallmatrix}
 \mathbf{c}\\\mathbf{d}
\end{smallmatrix}\right]\in S_+}
\theta(\left[\begin{smallmatrix}\mathbf{c}\\t\mathbf{d}\end{smallmatrix}\right],Z)^{4N(2^g-1)}}\vspace{0.1cm}\\
&=&
2^{4N}e\left(-2^gN(2^g-1)(2^g+1)\mathbf{r}^T(t\mathbf{s})\right)
\displaystyle\frac{ \prod_{\left[\begin{smallmatrix}
 \mathbf{a}\\\mathbf{b}
\end{smallmatrix}\right]\in S_-}
\theta(\left[\begin{smallmatrix}\mathbf{a}\\\mathbf{b}\end{smallmatrix}\right]-
\left[\begin{smallmatrix}\mathbf{r}\\t\mathbf{s}\end{smallmatrix}\right],Z)^{4N(2^g+1)}}
{\prod_{\left[\begin{smallmatrix}
 \mathbf{c}\\\mathbf{d}
\end{smallmatrix}\right]\in S_+}
\theta(\left[\begin{smallmatrix}\mathbf{c}\\\mathbf{d}\end{smallmatrix}\right],Z)^{4N(2^g-1)}}\vspace{0.1cm}\\
&&\textrm{by Lemma \ref{Igusa} (iii)}\\
&=&\Theta(\left[\begin{smallmatrix}\mathbf{r}\\t\mathbf{s}\end{smallmatrix}\right],Z).
\end{array}
\end{equation*}
\end{proof}

\begin{proposition}\label{bigtheta}
The function $\Theta(\left[\begin{smallmatrix}\mathbf{r}\\\mathbf{s}\end{smallmatrix}\right],Z)$
belongs to $\mathcal{F}_N$.
\end{proposition}
\begin{proof}
By Lemma \ref{Igusa} (v), $\Theta(\left[\begin{smallmatrix}\mathbf{r}\\\mathbf{s}\end{smallmatrix}\right],Z)$ belongs
to $\mathcal{F}_{2N^2}$.
\par
For any $\gamma\in\Gamma(N)$ we achieve that
\begin{eqnarray*}
\Theta(\left[\begin{smallmatrix}\mathbf{r}\\\mathbf{s}\end{smallmatrix}\right],\gamma(Z))&=&
\Theta(\gamma^T\left[\begin{smallmatrix}\mathbf{r}\\\mathbf{s}\end{smallmatrix}\right],Z)\quad
\textrm{by Lemma \ref{bigtransf} (ii)}\\
&=&\Theta(\left[\begin{smallmatrix}\mathbf{r}\\\mathbf{s}\end{smallmatrix}\right],Z)\quad
\textrm{by the fact $\gamma^T\left[\begin{smallmatrix}\mathbf{r}\\\mathbf{s}\end{smallmatrix}\right]
\equiv\left[\begin{smallmatrix}\mathbf{r}\\\mathbf{s}\end{smallmatrix}\right]
\Mod{\mathbb{Z}^{2g\phantom{\big|}\hspace{-0.1cm}}}$ and Lemma \ref{bigtransf} (i)}.
\end{eqnarray*}
This claims that $\Theta(\left[\begin{smallmatrix}\mathbf{r}\\\mathbf{s}\end{smallmatrix}\right],Z)$
lies in $\mathcal{A}_0(\Gamma(N),\mathbb{Q}(\zeta_{2N^2}))$.
\par
Let $s$ be an element of $\prod_p\mathbb{Z}_p^\times$ such that
$[s,\mathbb{Q}]$ is the identity on $\mathbb{Q}(\zeta_N)$. Take
a positive integer $t$ for which
\begin{equation*}
\iota(s)=\left[\begin{matrix}
I_g&O\\
O&s^{-1}I_g
\end{matrix}\right]\equiv\left[\begin{matrix}I_g&O\\O&tI_g\end{matrix}\right]\Mod
{2N^2\cdot M_{2g}(\mathbb{Z}_p)}\quad\textrm{for all rational primes $p$}.
\end{equation*}
Since $s_p\equiv1\Mod{N\cdot\mathbb{Z}_p}$ for all rational primes $p$, we have $t\equiv1\Mod{N}$.
We then obtain
\begin{eqnarray*}
\Theta(\left[\begin{smallmatrix}\mathbf{r}\\\mathbf{s}\end{smallmatrix}\right],Z)^{[s,\mathbb{Q}]}
&=&\Theta(\left[\begin{smallmatrix}\mathbf{r}\\\mathbf{s}\end{smallmatrix}\right],Z)^{\tau(\iota(s))}
\quad\textrm{by (A2)}\\
&=&\Theta(\left[\begin{smallmatrix}\mathbf{r}\\t\mathbf{s}\end{smallmatrix}\right],Z)
\quad\textrm{by Lemma \ref{bigtransf} (iii)}\\
&=&\Theta(\left[\begin{smallmatrix}\mathbf{r}\\\mathbf{s}\end{smallmatrix}\right],Z)
\quad\textrm{by the fact $t\equiv1\Mod{N}$ and Lemma \ref{bigtransf} (i)}.
\end{eqnarray*}
This implies that every Fourier coefficient of
$\Theta(\left[\begin{smallmatrix}\mathbf{r}\\\mathbf{s}\end{smallmatrix}\right],Z)$ lies in $\mathbb{Q}(\zeta_N)$.
\par
Therefore, we conclude that $\Theta(\left[\begin{smallmatrix}\mathbf{r}\\\mathbf{s}\end{smallmatrix}\right],Z)$ belongs to $\mathcal{F}_N$.
\end{proof}

\section {Siegel invariants}\label{seccm}

Let $n$ be a positive integer and $K$ be a CM-field with $[K:\mathbb{Q}]=2n$,
that is, $K$ is a totally imaginary quadratic extension of a totally real number field.
Fix a set $\{\varphi_1,\ldots,\varphi_n\}$ of embeddings of $K$ into $\mathbb{C}$ 
such that $\varphi_1,\ldots,\varphi_n,\overline{\varphi}_1,\ldots,\overline{\varphi}_n$
are all the embeddings of $K$ into $\mathbb{C}$,
which is called a CM-type of $K$.
Take a finite Galois extension $L$ of $\mathbb{Q}$ containing $K$ and set 
\begin{eqnarray*}
S&=&\{\sigma\in\mathrm{Gal}(L/\mathbb{Q})~|~\sigma|_K=\varphi_i~\textrm{ for some $1\leq i\leq n$}\},\\
S^*&=&\{\sigma^{-1}~|~\sigma\in S\},\\
H^*&=&\{\gamma\in\mathrm{Gal}(L/\mathbb{Q})~|~\gamma S^*=S^*\}.
\end{eqnarray*}
Let $K^*$ be the subfield of $L$ corresponding to the subgroup $H^*$ of $\mathrm{Gal}(L/\mathbb{Q})$ and $\{\psi_1,\ldots,\psi_g\}$ be the set of all the embeddings of $K^*$ into $\mathbb{C}$ obtained by the elements of $S^*$.
Then 
\begin{equation*}
K^*=\mathbb{Q}\left(\sum_{i=1}^n a^{\varphi_i}~|~a\in K\right)
\end{equation*}
and it is also a CM-field with a (primitive) CM-type $\{\psi_1,\ldots,\psi_g\}$ (\cite[Proposition 28 in $\S$8.3]{Shimura98}).
We define an embedding
\begin{eqnarray*}
\Psi~:~K^*&\rightarrow&\mathbb{C}^{g\phantom{\big|}\hspace{-0.1cm}}\\
a&\mapsto&\left[\begin{matrix}
a^{\psi_1}\\
\vdots\\
a^{\psi_g}
\end{matrix}\right].
\end{eqnarray*}
For an element $c$ of $K^*$ which is purely imaginary,  define an $\mathbb{R}$-bilinear form
$E_c:\mathbb{C}^{g\phantom{\big|}\hspace{-0.1cm}}\times\mathbb{C}^{g\phantom{\big|}\hspace{-0.1cm}}\rightarrow\mathbb{R}$ as
\begin{equation*}
E_c(\mathbf{u},\mathbf{v})=\sum_{j=1}^g c^{\psi_j}(u_j\overline{v}_j-\overline{u}_jv_j)
\quad(\mathbf{u}=\left[\begin{matrix}u_1\\\vdots\\u_g\end{matrix}\right],
\mathbf{v}=\left[\begin{matrix}v_1\\\vdots\\v_g\end{matrix}\right]\in\mathbb{C}^{g\phantom{\big|}\hspace{-0.1cm}}).
\end{equation*}
Then we know that
\begin{equation*}
E_c(\Psi(a),\Psi(b))=\mathrm{Tr}_{K^*/\mathbb{Q}}(ca\overline{b})
\quad\textrm{for all}~a,b\in K^*.
\end{equation*}

\begin{assumption}\label{pp}
We assume the following two conditions:
\begin{itemize}
\item[(C1)] The complex torus $\mathbb{C}^{g\phantom{\big|}\hspace{-0.1cm}}/\Psi(\mathcal{O}_{K^*})$ can be given a structure of a principally polarized abelian variety.
\item[(C2)] $(K^*)^*=K$.
\end{itemize}
\end{assumption}

\begin{remark}
\begin{itemize}
\item[(i)] It is well known that a complex torus can be equipped with a structure of an abelian variety if and only if there is a non-degenerate Riemann form on the torus in the sense of \cite[$\S$3.1]{Shimura98}.  See also \cite[$\S$4.2]{B-L}.
\item[(ii)] The assumption (C1) is equivalent to saying that there is an element $\xi$ of $K^*$
satisfying the following properties:
\begin{itemize}
\item[(P1)] $\xi^{\psi_i}$ lies on the positive imaginary axis
for every $1\leq i\leq g$.
\item[(P2)] The map $E_\xi$ yields a Riemann form on $\mathbb{C}^{g\phantom{\big|}\hspace{-0.1cm}}/\Psi(\mathcal{O}_{K^*})$.
\item[(P3)] $\delta_{K^*}^{-1}=\xi\mathcal{O}_{K^*}$, where $\delta_{K^*}$ is
the different ideal of $K^*$.
\end{itemize}
See \cite[Theorem 4 in $\S$6.2]{Shimura98}.
In this case, we call the pair $(\mathbb{C}^{g\phantom{\big|}\hspace{-0.1cm}}/\Psi(\mathcal{O}_{K^*}),E_\xi)$
a principally polarized abelian variety.
If the narrow class number of the maximal real subfield $F$ of $K^*$ is one, then the assumption (C1) is always true. 
Indeed, one can choose an element $\zeta$ of $K^*$ such that $K^*=F(\zeta)$ and $\zeta^{\psi_j}$ lies on the positive imaginary axis for every $1\leq j\leq g$ (\cite[p. 43]{Shimura98}).
Note that the different ideal $\delta_{K^*/F}$ of $K^*$ over $F$ is generated by the elements $\alpha-\overline{\alpha}$ for $\alpha\in\mathcal{O}_{K^*}$.
Since $\overline{\zeta}=-\zeta$, we see that $\zeta(\alpha-\overline{\alpha})\in F$ for every $\alpha\in\mathcal{O}_{K^*}$.
Hence $\zeta^{-1}\delta_{K^*}^{-1}=(\zeta\delta_{K^*/F})^{-1}\delta_{F}^{-1}$ is generated by an ideal of $F$.
Since the narrow class number of $F$ is one, $\zeta^{-1}\delta_{K^*}^{-1}=x\mathcal{O}_{K^*}$ for some totally positive element $x$ of $F^\times$.
Then $\xi=x\zeta$ satisfies (P1)$\sim$(P3).

\item[(iii)] The assumption (C2) holds if and only if $(K;\{\varphi_i\}_{i=1}^n)$ is a primitive CM-type, that is, the abelian varieties of this CM-type are simple (\cite[$\S$8.2, Proposition 26]{Shimura98}).
\item[(iv)] Throughout this paper, we fix an element $\xi$ of $K^*$ satisfying (P1)$\sim$(P3) so that $(\mathbb{C}^{g\phantom{\big|}\hspace{-0.1cm}}/\Psi(\mathcal{O}_{K^*}),E_\xi)$ becomes a principally polarized abelian variety.
\end{itemize}
\end{remark}

By Assumption \ref{pp} (C2), one can define a group homomorphism
\begin{equation*}
\begin{array}{cccc}
\varphi:&K^\times&\rightarrow& (K^*)^\times\\
&a&\mapsto&\displaystyle\prod_{i=1}^n a^{\varphi_i},
\end{array}
\end{equation*}
and extend it naturally
to a homomorphism of idele groups
$\varphi:K_\mathbb{A}^\times\rightarrow (K^*)_\mathbb{A}^\times$.
It is also known that for a fractional ideal $\mathfrak{a}$ of $K$ there
exists a fractional ideal $\varphi(\mathfrak{a})$ of $K^*$ such that
\begin{equation}\label{reflex norm}
\varphi(\mathfrak{a})\mathcal{O}_L=\prod_{i=1}^n(\mathfrak{a}\mathcal{O}_L)^{\varphi_i}
\end{equation}
(\cite[Proposition 29 in $\S$8.3]{Shimura98}).
\par
For a number field $F$ and a nonzero integral ideal $\mathfrak{a}$ of $F$ let $\mathcal{N}_F(\mathfrak{a})$
be the absolute norm of $\mathfrak{a}$, namely, $\mathcal{N}_F(\mathfrak{a})=|\mathcal{O}_F/\mathfrak{a}|$
(so, $\mathrm{N}_{F/\mathbb{Q}}(\mathfrak{a})=\mathcal{N}_F(\mathfrak{a})\mathbb{Z}$).
In general, for a fractional ideal $\mathfrak{b}$ of $F$ with
prime ideal factorization $\mathfrak{b}=\prod_{\mathfrak{p}}\mathfrak{p}^{e_\mathfrak{p}}$
we define
$\mathcal{N}_F(\mathfrak{b})=\prod_\mathfrak{p}\mathcal{N}_F(\mathfrak{p})^{e_\mathfrak{p}}$.
Furthermore, let
$D_{F/\mathbb{Q}}(\mathfrak{b})$ be the discriminant ideal of $\mathfrak{b}$ and
$d_{F/\mathbb{Q}}(\mathfrak{b})$ be its positive generator in $\mathbb{Q}$.
We then have the relation
\begin{equation*}
d_{F/\mathbb{Q}}(\mathfrak{b})=\mathcal{N}_F(\mathfrak{b})^2d_{F/\mathbb{Q}}(\mathcal{O}_F)
\end{equation*}
(\cite[Proposition 13 in Chapter III]{Lang86}).
\par
Let $K_0$ be the fixed field of $L$ by the subgroup
\begin{equation*}
\left\langle\sigma\in\mathrm{Gal}(L/\mathbb{Q})~|~\sigma|_K=\varphi_i~\textrm{for some}~i
\right\rangle
\end{equation*}
of $\mathrm{Gal}(L/\mathbb{Q})$. 
One can readily check that $K_0$ becomes either
an imaginary quadratic subfield of $K$ and $K^*$, or $\mathbb{Q}$.
In particular, we see from the assumption (C2) that $K_0=\mathbb{Q}$ when $g\geq 2$ (\cite[Remark (1) in p. 213]{Mumford} or \cite[Theorem 3 in $\S$6.2]{Shimura98}).
\par

From now on, we let $\mathfrak{f}=\mathfrak{f}_0\mathcal{O}_K$ for
a proper nontrivial ideal $\mathfrak{f}_0$ of $\mathcal{O}_{K_0}$.
Let $C$ be a given ray class in $\mathrm{Cl}(\mathfrak{f})$.
For an integral ideal $\mathfrak{c}$ in $C$ we set
\begin{equation*}
m_\mathfrak{c}=\sqrt[g]{\mathcal{N}_{K^*}((\mathfrak{f}^*)^{-1}\varphi(\mathfrak{c}))},
\end{equation*}
where $\mathfrak{f}^*=\mathfrak{f}_0\mathcal{O}_{K^*}$.
Let $d_0=2/[K_0:\mathbb{Q}]$.
Then we get 
\begin{equation}\label{mc}
\begin{array}{ccl}
m_\mathfrak{c}&=&\sqrt[g]{\mathcal{N}_{K^*}(\mathfrak{f}^*\overline{\mathfrak{f}^*})^{-1/2}\mathcal{N}_{K^*}(\varphi(\mathfrak{c})\overline{\varphi(\mathfrak{c})})^{1/2}}\vspace{0.1cm}\\
&=&\sqrt[g]{\mathcal{N}_{K^*}(\mathcal{N}_{K_0}(\mathfrak{f}_0)^{d_0}\mathcal{O}_{K^*})^{-1/2}
\mathcal{N}_{K^*}(\mathcal{N}_{K}(\mathfrak{c}))^{1/2}}\vspace{0.1cm}\\
&=&\mathcal{N}_{K_0}(\mathfrak{f}_0)^{-d_0}\mathcal{N}_K(\mathfrak{c}).
\end{array}
\end{equation}
Since $\mathfrak{f}^*\overline{\mathfrak{f}^*}=\mathcal{N}_{K_0}(\mathfrak{f}_0)^{d_0}\mathcal{O}_{K^*}$, one can  deduce from \cite[Lemma 5.3]{K-S-Y} that 
\begin{equation*}
\mathcal{P}_\mathfrak{c}=(\mathbb{C}^{g\phantom{\big|}\hspace{-0.1cm}}/\Psi(\mathfrak{f}^*\varphi(\mathfrak{c})^{-1}),E_{\xi m_\mathfrak{c}})
\end{equation*}
is also a principally polarized abelian variety.
Let $\{\mathbf{b}_1,\ldots,\mathbf{b}_{2g}\}$ be a symplectic basis of $\mathcal{P}_\mathfrak{c}$,
and let $y_1,\ldots,y_{2g}$ be elements of $\mathfrak{f}^*\varphi(\mathfrak{c})^{-1}$
satisfying $\mathbf{b}_j=\Psi(y_j)$ ($1\leq j\leq 2g$).
As is well known (\cite[Proposition 8.1.1]{B-L}), the $g\times g$ matrix
\begin{equation*}
Z_0^*=\left[\begin{matrix}\mathbf{b}_{g+1} & \cdots & \mathbf{b}_{2g}\end{matrix}\right]^{-1}
\left[\begin{matrix}\mathbf{b}_1 & \cdots & \mathbf{b}_g\end{matrix}\right]
\end{equation*}
belongs to $\mathbb{H}_g$, which we call a CM-point.
Since the smallest positive integer $N$ in $\mathfrak{f}=\mathfrak{f}_0\mathcal{O}_K$
also belongs to $\mathfrak{f}^*\varphi(\mathfrak{c})^{-1}=\mathfrak{f}_0\varphi(\mathfrak{c})^{-1}$, we can
express $N$ as
\begin{equation}\label{expressN}
N=\sum_{j=1}^{2g}r_j y_j
\quad\textrm{for some unique integers}~r_1,\ldots,r_{2g}.
\end{equation}

\begin{definition}\label{defSiegel}
We define the Siegel invariant $\Theta_\mathfrak{f}(C)$ modulo
$\mathfrak{f}$ at $C$ by
\begin{equation*}
\Theta_\mathfrak{f}(C)=\Theta(\left[\begin{matrix}r_1/N\\
\vdots\\
r_{2g}/N\end{matrix}\right],Z_0^*).
\end{equation*}
\end{definition}

\begin{remark}
When $g=1$,  Assumption \ref{pp} always holds and $\Theta_\mathfrak{f}(C)$ becomes the Siegel-Ramachandra invariant modulo $\mathfrak{f}$ at $C$ described in (\ref{S-R}).  

\end{remark}

This invariant $\Theta_\mathfrak{f}(C)$ is well defined, independent of the choices of
a symplectic basis of $\mathcal{P}_\mathfrak{c}$ and an integral ideal $\mathfrak{c}$ in $C$ as follows:

\begin{proposition}\label{indep2}
$\Theta_\mathfrak{f}(C)$ does not depend on the choice of a symplectic basis
 of $\mathcal{P}_\mathfrak{c}$.
\end{proposition}
\begin{proof}
Let $\{\widetilde{\mathbf{b}}_1,\ldots,\widetilde{\mathbf{b}}_{2g}\}$ be another symplectic basis of
$\mathcal{P}_\mathfrak{c}$ and let $\widetilde{Z}_0^*$ be the associated CM-point in $\mathbb{H}_g$.
Then we have
\begin{equation}\label{bbb}
\left[\begin{matrix}
\widetilde{\mathbf{b}}_1 & \cdots & \widetilde{\mathbf{b}}_{2g}
\end{matrix}\right]=
\left[\begin{matrix}
\mathbf{b}_1 & \cdots & \mathbf{b}_{2g}
\end{matrix}\right]\beta
\quad\textrm{for some}~
\beta=\left[\begin{matrix}A&B\\C&D\end{matrix}\right]\in\mathrm{Sp}_{2g}(\mathbb{Z}),
\end{equation}
and
\begin{equation}\label{ZZ}
\widetilde{Z}_0^*=\beta^T(Z_0^*)
\end{equation} (\cite[Proposition 6.1]{K-S-Y}).
\par

Let $\widetilde{y}_1,\ldots,\widetilde{y}_{2g}$ be elements of
$\mathfrak{f}^*\varphi(\mathfrak{c})^{-1}$ such that
$\widetilde{\mathbf{b}}_j=\Psi(\widetilde{y}_j)$ ($1\leq j\leq 2g$).
Together with (\ref{expressN}) we can express $N$ as
\begin{equation*}
N=\sum_{j=1}^{2g}r_jy_j=
\sum_{j=1}^{2g}\widetilde{r}_j\widetilde{y}_j\quad
\textrm{for some unique integers}~\widetilde{r}_1,\ldots, \widetilde{r}_{2g}.
\end{equation*}
Taking $\Psi$ we achieve by (\ref{bbb}) that
\begin{equation*}
\Psi(N)=\left[\begin{matrix}
\mathbf{b}_1 & \cdots & \mathbf{b}_{2g}
\end{matrix}\right]\left[\begin{matrix}r_1\\\vdots\\r_{2g}\end{matrix}\right]
=\left[\begin{matrix}
\widetilde{\mathbf{b}}_1 & \cdots & \widetilde{\mathbf{b}}_{2g}
\end{matrix}\right]\left[\begin{matrix}\widetilde{r}_1\\\vdots\\\widetilde{r}_{2g}\end{matrix}\right]
=\left[\begin{matrix}
\mathbf{b}_1 & \cdots & \mathbf{b}_{2g}
\end{matrix}\right]\beta\left[\begin{matrix}\widetilde{r}_1\\\vdots\\\widetilde{r}_{2g}\end{matrix}\right],
\end{equation*}
from which we have
\begin{equation}\label{rbr}
\left[\begin{matrix}
\widetilde{r}_1\\\vdots\\\widetilde{r}_{2g}
\end{matrix}\right]=
\beta^{-1}\left[\begin{matrix}
r_1\\\vdots\\r_{2g}
\end{matrix}\right].
\end{equation}
We then derive  that
\begin{eqnarray*}
\Theta(\left[\begin{matrix}\widetilde{r}_1/N\\\vdots\\
\widetilde{r}_{2g}/N\end{matrix}\right],\widetilde{Z}_0^*)&=&
\Theta(\beta^{-1}\left[\begin{matrix}r_1/N\\\vdots\\
r_{2g}/N\end{matrix}\right],\beta^T(Z_0^*))\quad
\textrm{by (\ref{ZZ}) and (\ref{rbr})}\\
&=&\Theta(\beta\beta^{-1}\left[\begin{matrix}r_1/N\\\vdots\\
r_{2g}/N\end{matrix}\right],Z_0^*)\quad\textrm{by
the fact $\beta^T\in\mathrm{Sp}_{2g}(\mathbb{Z})$ and Lemma \ref{bigtransf} (ii)}\\
&=&\Theta(\left[\begin{matrix}r_1/N\\\vdots\\
r_{2g}/N\end{matrix}\right],Z_0^*).
\end{eqnarray*}
This completes the proof.
\end{proof}

\begin{proposition}\label{indep1}
$\Theta_\mathfrak{f}(C)$ does not depend on the choice of an integral ideal $\mathfrak{c}$ in $C$.
\end{proposition}
\begin{proof}
Let $\mathfrak{c}'$ be another integral ideal in the class $C$, and so
\begin{equation*}
\mathfrak{c}'=\nu\mathfrak{c}\quad\textrm{for some}~\nu\in K^\times~
\textrm{such that}~\nu\equiv^*1\Mod{\mathfrak{f}}.
\end{equation*}
Then we may write $\nu$ as
\begin{equation*}
\nu=1+x\quad\textrm{for some}~x\in\mathfrak{f}\mathfrak{c}^{-1}.
\end{equation*}
Let
\begin{equation*}
\mathbf{b}_j'=\Psi(\varphi(\nu^{-1})y_j)
=\left[
\begin{matrix}
\varphi(\nu^{-1})^{\psi_1} & 0 & \cdots & 0\\
0 & \varphi(\nu^{-1})^{\psi_2} & \cdots & 0\\
\vdots & \vdots & \ddots & \vdots\\
0 & 0 & \cdots & \varphi(\nu^{-1})^{\psi_g}
\end{matrix}
\right]\mathbf{b}_j
\quad\textrm{for $1\leq j\leq 2g$}.
\end{equation*}
It then follows from the proof of \cite[Proposition 6.3]{K-S-Y} that
$\{\mathbf{b}_1',\ldots,\mathbf{b}_{2g}'\}$ is a symplectic basis of $\mathcal{P}_{\mathfrak{c}'}$, and the associated CM-point is 
\begin{equation*}
\left[\begin{matrix}\mathbf{b}_{g+1}' & \cdots
&\mathbf{b}_{2g}'\end{matrix}\right]^{-1}
\left[\begin{matrix}\mathbf{b}_1' & \cdots
&\mathbf{b}_g'\end{matrix}\right]
=\left[\begin{matrix}\mathbf{b}_{g+1} & \cdots
&\mathbf{b}_{2g}\end{matrix}\right]^{-1}
\left[\begin{matrix}\mathbf{b}_1 & \cdots
&\mathbf{b}_g\end{matrix}\right]=Z_0^*.
\end{equation*}
\par
Since $K_0$ is the subfield of $K$ fixed by the subgroup $\langle\sigma\in\mathrm{Gal}(L/\mathbb{Q})~|~
\sigma|_K=\varphi_i~\textrm{for some}~i\rangle$
of $\mathrm{Gal}(L/\mathbb{Q})$, we have
$(\mathfrak{f}\mathcal{O}_L)^{\varphi_i}=
(\mathfrak{f}_0\mathcal{O}_L)^{\varphi_i}=\mathfrak{f}_0\mathcal{O}_L=\mathfrak{f}^*\mathcal{O}_L$.
Hence we see from the fact $x\in\mathfrak{f}\mathfrak{c}^{-1}$ that
\begin{equation}\label{g (nu)}
\varphi(\nu)=\varphi(1+x)=\prod_{i=1}^n(1+x)^{\varphi_i}\in K^*\cap(1+\mathfrak{f}^*\varphi(\mathfrak{c})^{-1}\mathcal{O}_L)=
1+\mathfrak{f}^*\varphi(\mathfrak{c})^{-1}.
\end{equation}
Since $N\in\mathfrak{f}^*\varphi(\mathfrak{c}')^{-1}$ and $\{\varphi(\nu^{-1})y_1,\ldots,\varphi(\nu^{-1})y_{2g}\}$ is a $\mathbb{Z}$-basis for
$\mathfrak{f}^*\varphi(\mathfrak{c}')^{-1}$, one can express $N$ as
\begin{equation*}
N=\sum_{j=1}^{2g}r_j'\varphi(\nu^{-1})y_j\quad\textrm{for some integers}~
r_1',\ldots,r_{2g}'.
\end{equation*}
Hence we have
\begin{equation*}
\varphi(\nu)=\sum_{j=1}^{2g}(r_j'/N)y_j,
\end{equation*}
which implies by (\ref{expressN}) and (\ref{g (nu)})
\begin{equation*}
\left[\begin{matrix}r_1'/N\\\vdots\\r_{2g}'/N\end{matrix}\right]
\in\left[\begin{matrix}r_1/N\\\vdots\\r_{2g}/N\end{matrix}\right]
+\mathbb{Z}^{2g\phantom{\big|}\hspace{-0.1cm}}.
\end{equation*}
Therefore, the proposition follows from Lemma \ref{Igusa} (iii).
\end{proof}

\section {Galois conjugates of Siegel invariants}

Finally, we shall show that under the Assumption \ref{pp} the Siegel invariant $\Theta_\mathfrak{f}(C)$ lies in the ray class field $K_\mathfrak{f}$ and
satisfies the natural transformation formula via the Artin reciprocity map for $\mathfrak{f}$.
\par

Let $h:K^*\rightarrow M_{2g}(\mathbb{Q})$
be the regular representation with respect to the $\mathbb{Q}$-basis
$\{y_1,\ldots,y_{2g}\}$ of $K^*$, that is,
$h$ is the map given by the relation
\begin{equation}\label{hayay}
h(a)\left[\begin{matrix}
y_1 \\ \vdots \\ y_{2g}
\end{matrix}\right]=
a\left[\begin{matrix}
y_1 \\ \vdots \\ y_{2g}
\end{matrix}\right]\quad(a\in K^*).
\end{equation}
We naturally extend $h$ to the map $(K^*)_\mathbb{A}\rightarrow M_{2g}(\mathbb{Q}_\mathbb{A})$,
and also denote it by $h$.

\begin{proposition}[Shimura's Reciprocity Law]\label{reciprocity}
Let $f\in\mathcal{F}$. If $f$ is finite at $Z_0^*\in\mathbb{H}_g$, then
$f(Z_0^*)$ belongs to $K_\textrm{ab}$. Moreover, if $s\in K_\mathbb{A}^\times$, then
we get $h(\varphi(s))\in G_{\mathbb{A}+}$ and
\begin{equation*}
f(Z_0^*)^{[s,K]}=f^{\tau(h(\varphi(s)^{-1}))}(Z_0^*).
\end{equation*}
\end{proposition}
\begin{proof}
See \cite[Lemma 9.5 and Theorem 9.6]{Shimura00}.
\end{proof}

\begin{remark}
Observe that we are assuming $(K^*)^*=K$.
\end{remark}

\begin{theorem}\label{main}
If $\Theta_\mathfrak{f}(C)$ is finite, then it lies in $K_\mathfrak{f}$. Furthermore, it satisfies \begin{equation*}
\Theta_\mathfrak{f}(C)^{\sigma_\mathfrak{f}(D)}=
\Theta_\mathfrak{f}(CD)\quad\textrm{for all}~D\in\mathrm{Cl}(\mathfrak{f}),
\end{equation*}
where $\sigma_\mathfrak{f}$ is the Artin reciprocity map for $\mathfrak{f}$.
\end{theorem}
\begin{proof}
Since $\Theta_\mathfrak{f}(C)\in K_\textrm{ab}$ by Propositions \ref{bigtheta} and  \ref{reciprocity}, there is a positive integer $M$ such that $2N^2\,|\,M$ and $\Theta_\mathfrak{f}(C)\in K_\mathfrak{g}$, where $\mathfrak{g}=M\mathcal{O}_K$.
We can take integral ideals $\mathfrak{c}\in C$ and $\mathfrak{d}\in D$
which are relatively prime to $\mathfrak{g}$ by using the surjectivity of the natural map
$\mathrm{Cl}(\mathfrak{g})\rightarrow\mathrm{Cl}(\mathfrak{f})$.
Let $\{\mathbf{b}_1,\ldots,\mathbf{b}_{2g}\}$ and $\{\mathbf{d}_1,\ldots,\mathbf{d}_{2g}\}$
be symplectic bases of $\mathcal{P}_{\mathfrak{c}}$ and $\mathcal{P}_{\mathfrak{cd}}$, respectively.
Furthermore, let $y_1,\ldots,y_{2g}$ and $z_1,\ldots,z_{2g}$
be elements of $\mathfrak{f}^*\varphi(\mathfrak{c})^{-1}$ and
$\mathfrak{f}^*\varphi(\mathfrak{c}\mathfrak{d})^{-1}$, respectively,
such that $\mathbf{b}_j=\Psi(y_j)$ and $\mathbf{d}_j=\Psi(z_j)$
for $1\leq j\leq 2g$.
\par
Since $\mathfrak{f}^*\varphi(\mathfrak{c})^{-1}\subseteq
\mathfrak{f}^*\varphi(\mathfrak{c}\mathfrak{d})^{-1}=
\mathfrak{f}^*\varphi(\mathfrak{c})^{-1}\varphi(\mathfrak{d})^{-1}$, we have
\begin{equation}\label{ydz}
\left[\begin{matrix}
y_1\\\vdots\\y_{2g}
\end{matrix}\right]=\delta
\left[\begin{matrix}
z_1\\\vdots\\z_{2g}
\end{matrix}\right]\quad\textrm{for some}~\delta\in M_{2g}(\mathbb{Z})\cap\mathrm{GL}_{2g}(\mathbb{Q}),
\end{equation}
and hence
\begin{equation*}
\left[\begin{matrix}\mathbf{b}_1 & \cdots & \mathbf{b}_{2g}\end{matrix}\right]=
\left[\begin{matrix}\mathbf{d}_1 & \cdots & \mathbf{d}_{2g}\end{matrix}\right]\delta^T.
\end{equation*}
If we let $Z_0^*$ and $Z_1^*$ be the CM-points
associated with $\{\mathbf{b}_1,\ldots,\mathbf{b}_{2g}\}$ and
$\{\mathbf{d}_1,\ldots,\mathbf{d}_{2g}\}$, respectively,
then we obtain
\begin{equation}\label{ZdZ}
Z_1^*=\delta^{-1}(Z_0^*).
\end{equation}
We also obtain 
\begin{eqnarray*}
\left[\begin{matrix}
O&-I_g\\I_g&O
\end{matrix}\right]&=&
\left[\begin{matrix}
E_{\xi m_\mathfrak{c}}(\mathbf{b}_i,\mathbf{b}_j)
\end{matrix}\right]_{1\leq i,j\leq 2g}\\
&=&\delta\left[\begin{matrix}E_{\xi m_\mathfrak{c}}(\mathbf{d}_i,\mathbf{d}_j)
\end{matrix}\right]_{1\leq i,j\leq 2g}\delta^T\\
&=&\delta\left[\begin{matrix}m_\mathfrak{c}m_\mathfrak{cd}^{-1}E_{\xi m_\mathfrak{cd}}(\mathbf{d}_i,\mathbf{d}_j)
\end{matrix}\right]_{1\leq i,j\leq 2g}\delta^T\\
&=&m_\mathfrak{c}m_\mathfrak{cd}^{-1}\delta
\left[\begin{matrix}
O&-I_g\\I_g&O
\end{matrix}\right]\delta^T.
\end{eqnarray*}
This shows that
\begin{equation}\label{deltain}
\delta\in M_{2g}(\mathbb{Z})\cap G_+~
\textrm{with}~\det(\delta)=(m_\mathfrak{c}^{-1}m_\mathfrak{cd})^g=\mathcal{N}(\mathfrak{d})^g\quad
\textrm{by (\ref{mc})}.
\end{equation}
\par
Let $s=(s_p)_p$ be an idele of $K$ such that
\begin{equation}\label{idele}
\left\{\begin{array}{rccl}s_p&=&
1 & \textrm{if}~p\,|\,M,\\
s_p(\mathcal{O}_K)_p&=&\mathfrak{d}_p & \textrm{if}~p\nmid M.
\end{array}\right.
\end{equation}
If we denote by $\widetilde{D}$ the ray class in $\mathrm{Cl}(\mathfrak{g})$ containing $\mathfrak{d}$, then we get by (\ref{idele}) that
\begin{equation}\label{sKsD}
\begin{array}{rcl}
[s,K]|_{K_\mathfrak{g}}&=&\sigma_\mathfrak{g}(\widetilde{D}),\\
\quad \varphi(s)_p^{-1}(\mathcal{O}_{K^*})_p
&=&\varphi(\mathfrak{d})^{-1}_p\quad\textrm{for all rational primes $p$}.
\end{array}
\end{equation}
By (\ref{hayay})$\sim$(\ref{sKsD}), we deduce that for each rational prime $p$, the components of each of the vectors
\begin{equation*}
h(\varphi(s)^{-1})_p\left[\begin{matrix}y_1\\\vdots\\y_{2g}\end{matrix}\right]
\quad\textrm{and}\quad
\delta^{-1}\left[\begin{matrix}y_1\\\vdots\\y_{2g}\end{matrix}\right]
\end{equation*}
form a basis of $\mathfrak{f}^*\varphi(\mathfrak{c}\mathfrak{d})^{-1}_p=
\mathfrak{f}^*\varphi(\mathfrak{c})^{-1}\varphi(\mathfrak{d})^{-1}_p$. Thus there is a matrix $u=(u_p)_p\in\prod_p\mathrm{GL}_{2g}(\mathbb{Z}_p)$ satisfying
\begin{equation}\label{hgsud}
h(\varphi(s)^{-1})=u\delta^{-1}.
\end{equation}
\par
On the other hand, 
there exists a matrix $\gamma\in\mathrm{Sp}_{2g}(\mathbb{Z})$ such that
\begin{equation}\label{dIrM}
\delta\equiv\left[\begin{matrix}
I_g & O\\
O & \mathcal{N}(\mathfrak{d})I_g
\end{matrix}\right]\gamma\Mod{M\cdot M_{2g}(\mathbb{Z})}
\end{equation}
by (\ref{deltain}) and the surjectivity of the reduction $\mathrm{Sp}_{2g}(\mathbb{Z})
\rightarrow\mathrm{Sp}_{2g}(\mathbb{Z}/M\mathbb{Z})$ (\cite{Rapinchuk}).
Since $h(\varphi(s)^{-1})_p=I_{2g}$ for all $p\,|\,M$ by (\ref{idele}), we achieve
$u_p=\delta$ for all $p\,|\,M$ by (\ref{hgsud}). Hence we obtain by (\ref{dIrM}) that
\begin{equation}\label{urIMM}
u_p\gamma^{-1}\equiv\left[\begin{matrix}I_g&O\\
O&\mathcal{N}(\mathfrak{d})I_g\end{matrix}\right]\Mod{M\cdot M_{2g}(\mathbb{Z}_p)}
\end{equation}
for all rational primes $p$.
\par
If we write
\begin{equation*}
N=\sum_{j=1}^{2g}r_jy_j\quad\textrm{for some integers}~r_1,\ldots,r_{2g},
\end{equation*}
then we see by (\ref{ydz}) that
\begin{equation}\label{newvector}
N=\left[\begin{matrix}r_1 & \cdots & r_{2g}\end{matrix}\right]
\left[\begin{matrix}y_1\\\vdots\\y_{2g}\end{matrix}\right]
=(\left[\begin{matrix}r_1 & \cdots & r_{2g}\end{matrix}\right]\delta)
(\delta^{-1}\left[\begin{matrix}y_1\\\vdots\\y_{2g}\end{matrix}\right])
=(\left[\begin{matrix}r_1 & \cdots & r_{2g}\end{matrix}\right]\delta)
\left[\begin{matrix}z_1\\\vdots\\z_{2g}\end{matrix}\right].
\end{equation}
Letting $\mathbf{v}=\left[\begin{matrix}r_1/N\\\vdots\\r_{2g}/N\end{matrix}\right]$
we derive that
\begin{eqnarray*}
\Theta_\mathfrak{f}(C)^{\sigma_\mathfrak{g}(\widetilde{D})}
&=&\Theta_\mathfrak{f}(C)^{[s,K]}\quad\textrm{by (\ref{sKsD})}\\
&=&\Theta(\mathbf{v},Z_0^*)^{[s,K]}\quad\textrm{by Definition \ref{defSiegel}}\\
&=&\Theta(\mathbf{v},Z)^{\tau(h(\varphi(s)^{-1}))}|_{Z=Z_0^*}\quad\textrm{by Proposition \ref{reciprocity}}\\
&=&\Theta(\mathbf{v},Z)^{\tau(u\delta^{-1})}|_{Z=Z_0^*}\quad\textrm{by (\ref{hgsud})}\\
&=&\Theta(\mathbf{v},Z)^{\tau(u\gamma^{-1})\tau(\gamma)\tau(\delta^{-1})}|_{Z=Z_0^*}\\
&=&\Theta(\left[\begin{smallmatrix}
I_g&O\\
O&\mathcal{N}(\mathfrak{d}))I_g\end{smallmatrix}\right]\mathbf{v},Z)^{\tau(\gamma)\tau(\delta^{-1})}|_{Z=Z_0^*}
\quad\textrm{by (\ref{urIMM}) and Lemma \ref{bigtransf} (iii)}\\
&=&\Theta(\gamma^T\left[\begin{smallmatrix}
I_g&O\\
O&\mathcal{N}(\mathfrak{d})I_g\end{smallmatrix}\right]\mathbf{v},Z)^{\tau(\delta^{-1})}|_{Z=Z_0^*}
\quad\textrm{by Lemma \ref{bigtransf} (ii)}\\
&=&\Theta(\delta^T\mathbf{v},Z)^{\tau(\delta^{-1})}|_{Z=Z_0^*}
\quad\textrm{by (\ref{dIrM}) and Lemma \ref{bigtransf} (i)}\\
&=&\Theta(\delta^T\mathbf{v},\delta^{-1}(Z_0^*))\quad\textrm{owing to the fact
$\delta\in G_+$ and (A1)}\\
&=&\Theta_\mathfrak{f}(CD)\quad\textrm{by (\ref{ZdZ}), (\ref{newvector}) and Definition \ref{defSiegel}}.
\end{eqnarray*}
In particular, if $D$ is the identity class of $\mathrm{Cl}(\mathfrak{f})$ then $\sigma_\mathfrak{g}(\widetilde{D})$ leaves $\Theta_\mathfrak{f}(C)$ fixed.
Therefore, we conclude that $\Theta_\mathfrak{f}(C)$ lies in $K_\mathfrak{f}$ as desired.
\end{proof}

Lastly, we expect that under the Assumption \ref{pp} the following conjecture will turn out to be
affirmative.

\begin{conjecture}\label{conj}
The Siegel invariant $\Theta_\mathfrak{f}(C)$ discussed here is a primitive generator of the fixed field of
$\ker(\widetilde{\varphi})$ in the ray class field $K_\mathfrak{f}$ of a CM-field $K$, where $\widetilde{\varphi}:\mathrm{Cl}(\mathfrak{f})\rightarrow\mathrm{Cl}(\mathfrak{f^*})$ is the natural homomorphism induced from the map $\varphi$ defined in (\ref{reflex norm}). 
Here, $\mathrm{Cl}(\mathfrak{f}^*)$ is the ray class group of $K^*$ modulo $\mathfrak{f}^*$.
\end{conjecture}

\begin{example}\label{explicit example}
Let $\ell$ be an odd prime and $g=({\ell-1})/{2}$.
Let $K=\mathbb{Q}(\zeta_\ell)$ with $\zeta_\ell=e^{2\pi \mathrm{i}/\ell}$.
Then $[K:\mathbb{Q}]=2g$.
For each $1\leq i\leq g$, let $\varphi_i$ be
the element of $\mathrm{Gal}(K/\mathbb{Q})$
determined by $\zeta_\ell^{\varphi_i}=\zeta_\ell^i$.
Then $(K;\{\varphi_1^{-1},\varphi_2^{-1},\ldots,\varphi_g^{-1}\})$ is a primitive CM-type and its reflex is $(K;\{\varphi_1,\varphi_2,\ldots,\varphi_g\})$
(\cite[p. 64]{Shimura98}).
Furthermore, $(\mathbb{C}^{g\phantom{\big|}\hspace{-0.1cm}}/\Psi(\mathcal{O}_{K}), E_\xi)$ with $\xi=(\zeta_\ell-\zeta_\ell^{-1})/\ell$ becomes a principally polarized abelian variety, where  $\Psi(a)=\left[\begin{matrix}
a^{\varphi_1}\\
\vdots\\
a^{\varphi_g}
\end{matrix}\right]$ for $a\in K$ (\cite[p. 817--818]{K-Y1}).
Hence it satisfies the Assumption \ref{pp}.
\par
Assume that the class number of $K$ is $1$.
Let $\mathfrak{f}=N\mathcal{O}_K$ for a positive integer $N$, and let $C\in\mathrm{Cl}(\mathfrak{f})$.
Take an integral ideal $\mathfrak{c}$ of $K$ in $C$.
Then $\mathfrak{c}=\lambda\mathcal{O}_K$ for some $\lambda\in\mathcal{O}_K$.
Let 
\begin{equation*}
x_j = \left\{ \begin{array}{ll}
\zeta_\ell^{2j} & \textrm{for $1\leq j\leq g$}\\
\sum_{k=1}^{j-g}\zeta_\ell^{2k-1} & \textrm{for $g+1\leq j\leq 2g$},
\end{array} \right.
\end{equation*}
and $\varphi(a)=\prod_{i=1}^g a^{\varphi_i^{-1}}$ for $a\in K$.
Then $\{N\varphi(\lambda)^{-1}x_j\}_{j=1}^{2g}$ forms a free $\mathbb{Z}$-basis of $\mathfrak{f}\varphi(\mathfrak{c})^{-1}$ and
\begin{equation*}
\left[\begin{matrix}E_{\xi m_\mathfrak{c}}(\Psi(N\varphi(\lambda)^{-1}x_i),\Psi(N\varphi(\lambda)^{-1}x_j))\end{matrix}\right]_{1\leq i,j\leq 2g}=\left[\begin{matrix}O & -I_g\\
I_g&O\end{matrix}\right].
\end{equation*}
Thus $\{\Psi(N\varphi(\lambda)^{-1}x_j)\}_{j=1}^{2g}$ is a symplectic basis of $(\mathbb{C}^{g\phantom{\big|}\hspace{-0.1cm}}/\Psi(\mathfrak{f}\varphi(\mathfrak{c})^{-1}), E_{\xi m_\mathfrak{c}})$
and the corresponding CM-point is
\begin{equation}\label{cm point}
\begin{array}{ccl}
Z^*_\ell&=&
\left[\begin{matrix}
\Psi(N\varphi(\lambda)^{-1}x_{g+1})&\cdots&\Psi(N\varphi(\lambda)^{-1}x_{2g})
\end{matrix}\right]^{-1}
\left[\begin{matrix}
\Psi(N\varphi(\lambda)^{-1}x_1)&\cdots&\Psi(N\varphi(\lambda)^{-1}x_g)
\end{matrix}\right]\\
&=&\left[\begin{matrix}
\Psi(x_{g+1})&\cdots&\Psi(x_{2g})
\end{matrix}\right]^{-1}
\left[\begin{matrix}
\Psi(x_1)&\cdots&\Psi(x_g)
\end{matrix}\right].
\end{array}
\end{equation}
Note that $Z^*_\ell$ does not depend on a ray class $C$.
On the other hand, there exist integers $r_1,r_2,\ldots,r_{2g}$ such that
\begin{equation*}
N=\sum_{j=1}^{2g}r_j (N\varphi(\lambda)^{-1})x_j,
\end{equation*}
that is, $\varphi(\lambda)=\sum_{j=1}^{2g}r_j x_j$.
Then we obtain
\begin{equation}\label{value}
\Theta_\mathfrak{f}(C)=\Theta_{\left[\begin{smallmatrix}
r_1/N \\ r_2/N\\ \vdots \\ r_{2g}/N
\end{smallmatrix}\right]}(Z^*_\ell).
\end{equation}

Now, consider the case where $K=\mathbb{Q}(\zeta_5)$.

\begin{itemize}
\item[(i)] Let $\mathfrak{f}=5\mathcal{O}_K$.
One can readily show that $[K_\mathfrak{f} :K]=5$ and
\begin{equation*}
\mathrm{Cl}(\mathfrak{f})=\langle C_1 \rangle \cong \mathbb{Z}/5\mathbb{Z},
\end{equation*}
where $C_1$ denotes the ray class in $\mathrm{Cl}(\mathfrak{f})$ containing the ideal $(2+\zeta_5)\mathcal{O}_K$.
Let $C_k=C_1^k$ for an integer $k$.
Then we have
\begin{eqnarray*}
\varphi(2+\zeta_5)&=&(2+\zeta_5)(2+\zeta_5^3)\\
&=&-2\zeta_5-4\zeta_5^2-2\zeta_5^3-3\zeta_5^4\quad\textrm{since $\sum_{k=0}^4\zeta_5^k=0$}\\
&\equiv&x_1+2x_2+0\cdot x_3+3x_4 \Mod{\mathfrak{f}},
\end{eqnarray*}
where
\begin{equation*}
x_1=\zeta_5^2,\quad x_2=\zeta_5^4,\quad x_3=\zeta_5,\quad x_4=\zeta_5+\zeta_5^3.
\end{equation*}
Hence we see from (\ref{cm point}), (\ref{value})  and Lemma \ref{bigtransf} (i) that
\begin{equation*}
\Theta_\mathfrak{f}(C_1)=\Theta_{\left[\begin{smallmatrix}
1/5 \\ 2/5\\0 \\ 3/5
\end{smallmatrix}\right]}(Z^*_5)\approx-2.13359\times 10^{-69}+4.17297\times 10^{-70}\mathrm{i},
\end{equation*}
where 
\begin{eqnarray*}
Z_5^*=
\left[\begin{matrix}
\zeta_5&\zeta_5+\zeta_5^3\\
\zeta_5^2&\zeta_5^2+\zeta_5
\end{matrix}\right]^{-1}
\left[\begin{matrix}
\zeta_5^2&\zeta_5^4\\
\zeta_5^4&\zeta_5^3
\end{matrix}\right].
\end{eqnarray*}
In like manner, we obtain
\begin{equation*}
\begin{array}{ccccl}
\Theta_\mathfrak{f}(C_2)&=&\Theta_{\left[\begin{smallmatrix}
3/5 \\ 1/5\\3/5 \\ 2/5
\end{smallmatrix}\right]}(Z^*_5)&\approx&(4.16089-1.58401\mathrm{i})\times 10^{-50}\\
\Theta_\mathfrak{f}(C_3)&=&\Theta_{\left[\begin{smallmatrix}
3/5 \\ 0\\2/5 \\ 2/5
\end{smallmatrix}\right]}(Z^*_5)&\approx&(4.16089+1.58401\mathrm{i})\times 10^{-50}\\
\Theta_\mathfrak{f}(C_4)&=&\Theta_{\left[\begin{smallmatrix}
2/5 \\ 2/5\\4/5 \\ 4/5
\end{smallmatrix}\right]}(Z^*_5)&\approx&-2.13359\times 10^{-69}-4.17297\times 10^{-70}\mathrm{i}\\
\Theta_\mathfrak{f}(C_5)&=&\Theta_{\left[\begin{smallmatrix}
1/5 \\ 1/5\\0 \\ 1/5
\end{smallmatrix}\right]}(Z^*_5)&\approx&4.85930\times 10^{-254}.
\end{array}
\end{equation*}
Here we estimate these values with the aid of Maple software.
Observe that
\begin{equation*}
\Theta_\mathfrak{f}(C_k)=\overline{\Theta_\mathfrak{f}(C_{5-k})}\quad\textrm{for $k\in\mathbb{Z}$}.
\end{equation*}
Since all conjugates of $\Theta_\mathfrak{f}(C_1)$ are distinct, we get
\begin{equation*}
K_\mathfrak{f}=K(\Theta_\mathfrak{f}(C))\quad\textrm{for $C\in\mathrm{Cl}(\mathfrak{f})$} .
\end{equation*}

\item[(ii)] Let $\mathfrak{f}=6\mathcal{O}_K$.
Then we have $[K_\mathfrak{f} :K]=10$ and
\begin{equation*}
\mathrm{Cl}(\mathfrak{f})=\langle C_1 \rangle \cong \mathbb{Z}/10\mathbb{Z}.
\end{equation*}
In a similar way as in (i), one can estimate
\begin{equation*}
\begin{array}{ccccl}
\Theta_\mathfrak{f}(C_1)&=&\Theta_{\left[\begin{smallmatrix}
2/6 \\ 3/6\\0 \\ 4/6
\end{smallmatrix}\right]}(Z^*_5)&\approx&(-1.68219-1.88870\mathrm{i})\times 10^{-66}\\
\Theta_\mathfrak{f}(C_2)&=&\Theta_{\left[\begin{smallmatrix}
2/6 \\ 5/6\\0 \\ 2/6
\end{smallmatrix}\right]}(Z^*_5)&\approx&(9.08964+7.01165\mathrm{i})\times 10^{-135}\\
\Theta_\mathfrak{f}(C_3)&=&\Theta_{\left[\begin{smallmatrix}
3/6 \\ 4/6\\4/6 \\ 4/6
\end{smallmatrix}\right]}(Z^*_5)&\approx&(-3.16257+1.88358\mathrm{i})\times 10^{-65}\\
\Theta_\mathfrak{f}(C_4)&=&\Theta_{\left[\begin{smallmatrix}
5/6 \\ 3/6\\0 \\ 3/6
\end{smallmatrix}\right]}(Z^*_5)&\approx&(2.29176+1.51419\mathrm{i})\times 10^{-93}\\
\end{array}
\end{equation*}
\begin{equation*}
\begin{array}{ccccl}
\Theta_\mathfrak{f}(C_5)&=&\Theta_{\left[\begin{smallmatrix}
5/6 \\ 0\\0 \\ 4/6
\end{smallmatrix}\right]}(Z^*_5)&\approx&8.33316\times 10^{-136}\\
\Theta_\mathfrak{f}(C_6)&=&\Theta_{\left[\begin{smallmatrix}
5/6 \\ 2/6\\0 \\ 2/6
\end{smallmatrix}\right]}(Z^*_5)&\approx&(2.29176-1.51419\mathrm{i})\times 10^{-93}\\
\Theta_\mathfrak{f}(C_7)&=&\Theta_{\left[\begin{smallmatrix}
4/6 \\ 2/6\\1/6 \\ 3/6
\end{smallmatrix}\right]}(Z^*_5)&\approx&(-3.16257-1.88358\mathrm{i})\times 10^{-65}\\
\Theta_\mathfrak{f}(C_8)&=&\Theta_{\left[\begin{smallmatrix}
3/6 \\ 1/6\\0 \\ 3/6
\end{smallmatrix}\right]}(Z^*_5)&\approx&(9.08964-7.01165\mathrm{i})\times 10^{-135}\\
\Theta_\mathfrak{f}(C_9)&=&\Theta_{\left[\begin{smallmatrix}
5/6 \\ 3/6\\0 \\ 1/6
\end{smallmatrix}\right]}(Z^*_5)&\approx&(-1.68219+1.88870\mathrm{i})\times 10^{-66}\\
\Theta_\mathfrak{f}(C_{10})&=&\Theta_{\left[\begin{smallmatrix}
5/6 \\ 5/6\\0 \\ 5/6
\end{smallmatrix}\right]}(Z^*_5)&\approx&3.26284\times 10^{-348},
\end{array}
\end{equation*}
which are distinct. 
Note that
\begin{equation*}
\Theta_\mathfrak{f}(C_k)=\overline{\Theta_\mathfrak{f}(C_{10-k})}\quad\textrm{for $k\in\mathbb{Z}$}.
\end{equation*}
Therefore, we conclude
\begin{equation*}
K_\mathfrak{f}=K(\Theta_\mathfrak{f}(C))\quad\textrm{for every $C\in\mathrm{Cl}(\mathfrak{f})$} .
\end{equation*}

\end{itemize}
\end{example}

\bibliographystyle{amsplain}

\address{
Department of Mathematical Sciences \\
KAIST \\
Daejeon 34141\\
Republic of Korea} {jkkoo@math.kaist.ac.kr}
\address{
Laboratoire de Math\'ematiques\\
Institut Fourier\\
B.P. 74\\
F-38402  Saint-Martin-d'H\`eres\\
France} {gilles.rbrt@gmail.com}
\address{
Department of Mathematics\\
Hankuk University of Foreign Studies\\
Yongin-si, Gyeonggi-do 17035\\
Republic of Korea} {dhshin@hufs.ac.kr}
\address{
Department of Mathematical Sciences\\
Ulsan National Institute of Science and Technology\\
Ulsan 44919\\
Republic of Korea} {dsyoon@unist.ac.kr}

\end{document}